\documentclass[10pt]{gtpart} 
\usepackage{amsmath, amsthm, amssymb, amscd}
\usepackage{mathtools}
\usepackage{graphicx}
\usepackage[all]{xy}
\xyoption{all}

\title{On the Hochschild and cyclic (co)homology of rapid decay group algebras}

\author{Ronghui Ji}
\address{
		Department of Mathematical Sciences\\
		Indiana University-Purdue University, Indianapolis\\ 
		402 N. Blackford Street\\
		Indianapolis, IN 46202\\
		}
\email{ronji@math.iupui.edu}

\author{Crichton Ogle}
\address{
		Department of Mathematics\\
		The Ohio State University\\
        Columbus, OH 43210\\
        }
\email{ogle@math.ohio-state.edu}

\author{Bobby W. Ramsey}
\address{
		Department of Mathematics\\
		University of Hawai`i at M\=anoa\\ 
		2565 McCarthy Mall\\
		Honolulu, HI 96822\\
		}
\email{bwramsey@math.hawaii.edu}

\newtheorem{theorem}{Theorem}

\newtheorem{lemma}{Lemma}
\newtheorem{corollary}{Corollary}

\newtheorem*{conjSrBC}{Conjecture - $\B$-SrBC }
\newtheorem*{conjl1SrBC}{Conjecture - $\ell^1$-SrBC }
\newtheorem*{thmA}{Theorem A}
\newtheorem*{thmB}{Theorem B}

\theoremstyle{definition}
\newtheorem{definition}{Definition}
\newtheorem{example}{Example}

\newtheorem*{nnremark}{Remark}

\newcommand{\isom}{{\,\cong\,}}
\newcommand{\tensor}{\otimes}
\newcommand{\btensor}{{\hat{\otimes}}}

\newcommand{\bHom}{\operatorname{Hom}^{bdd}}
\newcommand{\Hom}{\operatorname{Hom}}
\newcommand{\Ext}{\operatorname{Ext}}

\newcommand{\bExt}{\operatorname{bExt}}

\newcommand{\Dist}{\operatorname{Dist}}

\newcommand{\N}{\mathbb{N}}
\newcommand{\E}{\mathcal{E}}
\newcommand{\BE}{\mathcal{B}\textrm{-}\mathcal{E}}
\newcommand{\B}{\mathcal{B}}
\newcommand{\cP}{\mathcal{P}}

\newcommand{\BG}{{\mathcal{H}_{\B,L}\left(G\right)}}
\newcommand{\CG}{{\mathbb{C}[G]}}

\newcommand{\RH}{\mathcal{H}}

\newcommand{\surj}{\twoheadrightarrow}
\newcommand{\inj}{\rightarrowtail}

\allowdisplaybreaks[4]
\begin{document}
\begin{abstract} We show that the technical condition of solvable conjugacy bound, introduced in \cite{JOR1}, can be removed without affecting the main results of that paper. The result is a Burghelea-type description of the summands $HH_*^t(\BG)_{<x>}$ and $HC_*^t(\BG)_{<x>}$ for any bounding class $\B$, discrete group with word-length $(G,L)$ and conjugacy class $<x>\in <G>$. We use this description to prove the conjecture $\B$-SrBC of \cite{JOR1} for a class of groups that goes well beyond the cases considered in that paper. In particular, we show that the conjecture $\ell^1$-SrBC (the Strong Bass Conjecture for the topological $K$-theory of $\ell^1(G)$) is true for all semihyperbolic groups which satisfy SrBC, a statement consistent with the rationalized Bost conjecture for such groups.
\end{abstract}
\maketitle
\vspace{.5in}

\tableofcontents


\section{Introduction} 
Given a bounding class $\B$ (the definition of which we recall below) and group with word-length function $(G,L)$, one may 
construct the rapid decay algebra $\BG$. This algebra was introduced in \cite{JOR1}; it is a Fr\'echet algebra whenever 
the bounding class $\B$ is equivalent to a countable class (there are no known cases when this doesn't occur). When $\B = \cP$, 
the class of polynomial bounding functions, this algebra is precisely the rapid decay algebra $H^{1,\infty}_L(G)$ introduced 
by Jolissaint in \cite{Jo1}, \cite{Jo2}. Because $\BG$ is a rapid decay algebra formed using weighted $\ell^1$-norms one has, 
associated to each conjugacy class $<x>\in <G>$, summands $F_*^t(\BG)_{<x>}$ of $F_*^t(\BG)$, where 
$F = HH, HC, HPer$, and thus projection maps
\[
p_{<x>}:F_*^t(\BG)\surj F_*^t(\BG)_{<x>},\quad F = HH, HC, HPer
\]
and similarly for cohomology.

Recall that a conjugacy class {$<x>$} is {\it elliptic} if $<x>$ has finite order, and {\it non-elliptic} if $x$ has infinite order. 
As in \cite{JOR1}, one may then posit - for a given bounding class $\B$ and group with word-length $(G,L)$ - the following 
generalization of the Strong Bass Conjecture:
\begin{conjSrBC}
For each non-elliptic conjugacy class $<x>$, the image of the composition
\[
p_{<x>}\circ ch_*:K^t_*(\BG)\to HC^t_*(\BG)\surj HC^t_*(\BG)_{<x>}
\]
is zero.
\end{conjSrBC}

When $\B = \cP$, the subalgebra $\RH_{{\cP},L}(G) = H^{1,\infty}L(G)$ is smooth in $\ell^1(G)$ by \cite{Jo1}, \cite{Jo2}; 
in this case the above conjecture can be restated as:
\begin{conjl1SrBC}
For each non-elliptic conjugacy class $<x>$, the image of the composition
\[
p_{<x>}\circ ch_*:K^t_*(\ell^1(G))\to HC^t_*(H^{1,\infty}_L(G))\surj
HC^t_*(H^{1,\infty}_L(G))_{<x>}
\]
is zero.
\end{conjl1SrBC}

The conjecture $\ell^1$-SrBC is certainly the most important special case of the first more general 
conjecture; as shown in the appendix of \cite{JOR1} these conjectures follow from the rational surjectivity 
of an appropriately defined Baum-Connes assembly map. Thus any counter-example to one of these conjectures 
would, in turn, provide a counterexample to a Baum-Connes type conjecture holding for the corresponding rapid decay algebra.

It is a result due to Burghelea that one has decompositons
\[
HH_*(\CG)\isom \underset{<x>\in <G>}{\oplus} HH_*(\CG)_{<x>},\quad
HC_*(\CG)\isom \underset{<x>\in <G>}{\oplus} HC_*(\CG)_{<x>}
\]
and for each conjugacy class $<x>$ an isomorphism
\[
HH_*(\CG)_{<x>}\isom H_*(BG_x;\C)
\]
where $G_x$ denotes the centralizer of $x\in G$. This last fact allows one to derive corresponding descriptions of 
$HC_*(\CG)_{<x>}$; specifically, for non-elliptic classes one has
\[
HC_*(\CG)_{<x>}\isom H_*(B(G_x/(x));\C)
\]
where $(x)\subset G_x$ denotes the infinite cyclic subgroup generated by $x$, and $G_x/(x)$ the quotient group. The 
standard way of arriving at this result is via the isomorphism of sets
\begin{equation}\label{eqn:fundmap}
G_x\backslash G\to S_{<x>}
\end{equation}
where $S_{<x>}\subset G$ is the set of elements conjugate to $x$, and the correspondence is given by $G_xg\mapsto g^{-1}xg$. 
When $G$ is equipped with word-length $L$, this map becomes a $\B$-G-morphism of weighted right $\B$-G-sets in the sense of 
\cite{JOR1}, but one whose inverse is in general unbounded. This leads to the notion of $\B$-solvable conjugacy bound, 
introduced in \cite{JOR1} - the condition that the map in (\ref{eqn:fundmap}) is a $\B$-G-isomorphism - which in turn allows 
for the identification of the summands $F_*^t(\BG)_{<x>}$ in terms of the rapid decay homology of $BG_x$ or 
$B(G_x/(x))$. Unfortunately, this condition is very restrictive, as it requires the group to have a conjugacy problem which 
is solvable in $\B$-bounded time. In particular, when $\B = \cP$, it is difficult to go much beyond hyperbolic groups with 
this restriction in place (the main result of \cite{JOR1} was verification of a relative version of the $\ell^1$-BC conjecture 
in the presence of relative hyperbolicity).

With this in mind, we may now state our main result, by which the solvable conjugacy bound constraint of \cite{JOR1} is removed. 
The consequence is a complete characterization of the conjugacy class summands of the Hochschild and cyclic (co)-homology groups 
of $\BG$ for any group with word-length $(G,L)$ and bounding class $\B$.  Recall that the \underline{distortion} of a subgroup with
length function $(H, L_H)$ of a group with length function $(G,L_G)$ is the function $Dist(H) : \N \to \N$ defined by
$Dist(H)(n) = \max \left\{ L_H(h) \, | \, h \in H, L_H(h) \leq n \right\}$.

\begin{thmA}
Let $(G,L)$ be a discrete group equipped with proper word-length function and $\B$ a countable bounding class. Then for each 
conjugacy class $<x>\in <G>$ there are isomorphisms in topological Hochschild (co)-homology
\begin{gather*}
HH^*_t(\BG)_{<x>} \isom {\B}H^*(BG_x),\quad HH_*^t(\BG)_x \isom {\B}H_*(BG_x)
\end{gather*}
These isomorphisms imply the existence of isomorphisms in topological cyclic (co)-homology
\begin{gather*}
	\begin{rcases}
		HC^*_t(\BG)_{<x>} \isom {\B}H^*(BG_x)\otimes HC^*(\C)\\
		HC_*^t(\BG)_{<x>} \isom {\B}H_*(BG_x)\otimes HC_*(\C)
	\end{rcases} \quad (<x>\text{ elliptic})\\
HC^*_t(\BG)_{<x>} \isom {\B}H^*(B(G_x/(x))),\quad <x>\text{ non-elliptic and } \Dist((x))\leq \B \\
HC_*^t(\BG)_{<x>} \isom
{\B}H_*(B(G_x/(x))),\quad <x>\text{ non-elliptic and } \Dist((x))\leq \B
\end{gather*}
\end{thmA}
\footnotetext{The statement of this theorem amends Theorem 1.4.7 of \cite{JOR1} where the effects of distortion were not
taken into account.  }
Here the weighting on $BG_x$ and $B(G_x/(x))$ comes from the induced word-length function on $G_x$ and the corresponding quotient 
word-length function on $G_x/(x)$, while ${\B}H^*(_-)$, ${\B}H_*(_-)$ denote the $\B$-bounded (co)-homology groups as defined in 
\cite{JOR2} (and reviewed below). For the cyclic groups, the $\C[u]$-module and comodule structures of the terms on the right 
are exactly as in the case of the group algebra for both elliptic and non-elliptic classes (compare \cite{B1}). Finally, $\Dist((x))$ 
refers to the distortion of the infinite cyclic subgroup $(x)$ of $G$; the condition $\Dist((x))\leq \B$ means that the distortion of 
$(x)$ is bounded above by a function in $\B$.

This theorem allows for the construction of a class of groups satisfying the $\B$-SrBC conjecture, modeled on the class of groups 
originally considered by the first author in \cite{Ji1} and independently by Chadha and Passi in \cite{CP}, and extended to a slightly larger class by Emmanouil in \cite{E1}. It also 
allows for direct verification of the $\ell^1$-SrBC conjecture for a large and important class of groups, as given by the following 
theorem. Recall that a group $G$ satisfies the \underbar{nilpotency condition} 
if $S:HC^*(\CG)_{<x>}\to HC^{*+2}(\CG)_{<x>}$ is a nilpotent operator for all non-elliptic classes $<x>\in <G>$. 
Similarly, $G$ satisfies the \underbar{$\B$-nilpotency condition} if $S:HC^*_t(\BG)_{<x>}\to HC^{*+2}(\BG)_{<x>}$ 
is a nilpotent operator for all non-elliptic classes $<x>\in <G>$. By \cite{JOR1}, if $G$ satisfies the $\B$-nilpotency condition, then 
$\B$-SrBC is true for $G$. Finally, $(G,L)$ is $\B$-isocohomological if the comparison map in cohomology ${\B}H^*(BG)\to H^*BG)$ is an 
isomorphism (\cite{JOR2}).

\begin{thmB}
Let $(G,L)$ be a semihyperbolic group in the sense of \cite{AB}.  If $G$ satisfies the nilpotency condition, then it satisfies the
$\B$-nilpotency condition for any bounding class $\B$ containing $\cP$. In particular, if $G_x/(x)$ has finite cohomogical 
dimension over $\Q$ for each non-elliptic class $<x>$, then $G$ satisfies $\B$-SrBC for all bounding classes $\B$ containing $\cP$.
\end{thmB}

\begin{proof}
By \cite{AB} the centralizer subgroups $G_h$ are quasi-convex subgroups
of the (linearly) combable group $G$, hence each $G_h$ is combable
and quasi-isometrically embedded into $G$. By \cite{O1}, \cite{M1}, or \cite{JOR2} each $G_h$ is $\B$-isocohomological
for any bounding class containing $\cP$. Moreover, $(x)$ is quasi-isometrically embedded in $G_x$ - hence also in $G$ - for each conjugacy class $<x>\in <G>$, implying its distortion in $G$ is at most polynomial. Thus for each non-elliptic conjugacy
class $<h> \in <G>$ we have isomorphisms
\[
HC^*_t(\BG)_{<h>} \isom {\B}H^*(B(G_h/(h))) \isom H^*( B(G_h/(h)) )
\]
By Burghelea's result together with Theorem A above, the nilpotency condition for $G$ implies the $\B$-nilpotency condition for $(G,L)$ whenever $\B$ is a bounding class containing $\cP$.
\end{proof}
\newpage


\section{Preliminaries}

Let $\mathcal{S}$ denote the set of non-decreasing functions $\{ f : \R_+ \to \R^+\}$. Suppose
$\phi : \mathcal{S}^n \to \mathcal{S}$ is a function of sets, and $\B \subset \mathcal{S}$. $\B$ is
weakly closed under $\phi$ if for each $(f_1, f_2, \ldots, f_n) \in \B^n$ there is $f \in \B$ with
$\phi( f_1, f_2, \ldots, f_n ) \leq f$. A bounding class is a subset $\B \subset \mathcal{S}$ satisfying
\begin{enumerate}
	\item It contains the constant function $1$.
	\item It is weakly closed under the operations of taking positive rational linear combinations.
	\item It is weakly closed under the operation $(f,g) \mapsto f \circ g$ for $g \in \mathcal{L}$,
		where $\mathcal{L}$ denotes the linear bounding class $\{ f(x) = ax + b \, | \, a,b \in \Q_+\}$.
\end{enumerate}
A bounding class is \underline{composable} if it is weakly closed under the operation of composition of fuctions.

Naturally occurring classes besides $\mathcal{L}$ are $\B_{min} = \Q_+$, $\cP =$ the set of polynomials
with non-negative rational coefficients, and $\mathcal{E} = \{ e^f \, | \, f \in \mathcal{L} \}$.
Bounding classes were discussed in detail in \cite{JOR1, JOR2}.   The bounding classes we will be 
considering will all contain $\mathcal{L}$. 

Recall that a function $f : X \to \R_+$, with $X$ discrete, is \underline{proper} if
the preimage of any bounded subset is finite.
A \underline{weighted set} $(X,w)$ will refer to a countable
discrete set $X$ together with a proper function $w:X\to\R_+$.
A morphism between weighted sets $\phi:(X,w)\to (X',w')$ is \underline{$\B$-bounded} if
\[
\forall f\in\B \, \exists f'\in\B \text{ such that } f(w'(\phi(x)))\leq f'(w(x))\qquad\forall x\in X.
\]
Given two weighted sets $(X,w)$ and $(X',w')$, the Cartesian product $X \times X'$ carries a canonical 
weighted set structure, $(X \times X', r )$ with weight function $r( x,x' ) = w(x) + w'(x')$.
Suppose further that $X$ is a right $G$-set and $X'$ is a left $G$-set.  Let $X \underset{G}{\times} X'$
denote the quotient of $X \times X'$ by the relation $( x, gx' ) \sim (xg, x')$.  There is a natural
weight, $\rho$, on $X \underset{G}{\times} X'$ induced by the weight $r$ on $X \times X'$ as follows.
\[ \rho( [ x,x' ] ) = \min_{g \in G} r( xg, g^{-1} x' ). \]

Given a weighted set $(X,w)$ and $f \in \mathcal{S}$, the seminorm $| \cdot |_f$ on $\Hom(X,\C)$ is given
by $| \phi |_f := \sum_{x\in X} | \phi(x) | f( w(x) )$. We will mainly be concerned with the case
$(X,w) = (G, L)$ is a discrete group endowed with a length function $L$. The length function $L$
is called a word-length function ( with respect to a generating set $S$ ) if $L(1) = 1$ and there is
a function $\phi : S \to \R^+$ with
\[ L(g) = \min\left\{ \sum_{i=1}^n \phi(x_i) \, | \, x_i \in S, \, x_1 x_2 \ldots x_n = g \right\} . \]
When $S$ is finite, taking $\phi = 1$ produces the standard word-length function on $G$.

Recall that a simplicial object in a category $\mathcal{C}$ is a covariant functor $B_{\bullet} : \Delta \to \mathcal{C}$,
from the simplicial category into $\mathcal{C}$. That is, for each $k \geq 0$ $B_k \in \mathcal{C}^{(0)}$,
and there exist degeneracy morphisms in $\Hom_{\mathcal{C}}(B_k, B_{k+1})$ and
face morphisms in $\Hom_{\mathcal{C}}(B_k, B_{k-1})$ which satisfy the usual simplicial identities.

The augmented simplicial category, $\Delta^+$, is obtained by adding another object $[-1]$ to $\Delta$,
and a single face morphism in $\Hom_{\Delta^{+}}([-1], [0])$. An augmented simplicial object in a category
$\mathcal{C}$ is a covariant functor $B_{\bullet} : \Delta^+ \to \mathcal{C}$. This consists of a
simplicial set $B_k$, $k \geq 0$, as well as $B_{-1} \in \mathcal{C}^{(0)}$ with an augmentation
$B_0 \surj B_{-1}$.

An augmented simplicial group with word-length $(\Gamma^+_{\bullet}, L_{\bullet})$ consists of an augmented simplicial
group $\Gamma^+_{\bullet}$, with $L_k$ a word-length function on $\Gamma_k$ for all $k \geq -1$. If $\B$
is a bounding class, the augmented simplicial group with word-length $(\Gamma^+_{\bullet}, L_{\bullet})$ is $\B$-bounded
if all face and degeneracy maps, including the augmentation map, are $\B$-bounded with respect to the
word-lengths, $\{L_k\}_{k \geq -1}$. Then $(\Gamma^+_{\bullet}, L_{\bullet})$ is a \underbar{type $\B$-resolution} if
i) $(\Gamma^+_k, L_k)$ is a countably generated free group with $\N$-valued word-length metric $L_k$
generated by a proper function on the set of generators for $\Gamma_k$ for all $k \geq 0$, and
ii) $\Gamma^+_{\bullet}$ admits a simplicial set contraction $\tilde{s} = \{ \tilde{s}_{k+1}: \Gamma_k \to \Gamma_{k+1} \}_{k \geq -1}$
which is a $\B$-bounded set map in each degree. Every countable discrete group $G$ admits a type
$\B$ resolution $\Gamma^+_{\bullet}$ with $G = \Gamma_{-1}$. In fact, starting with $G$, the
resolution can always be constructed so that the face and degeneracy maps, the augmentation, and
the simplicial contraction are all linearly bounded \cite[Appendix]{O1}.

\begin{example}
The cyclic bar construction, $N^{cy}_{\bullet}(G)$.
\end{example}
Let $G$ be a countable group and let $L$ be a word-length function on $G$.
We define $N_{\bullet}^{cy}(G) = \{[n]\mapsto G^{n+1}\}_{n\geq 0}$ with
\begin{align*}
	\partial_i(g_0,g_1,\dots,g_n) &= (g_0,g_1,\dots,g_ig_{i+1},\dots,g_n),\, 0 \leq i \leq n-1,\\
	\partial_n(g_0,g_1,\dots,g_n) &= (g_n g_0,g_1,g_2,\dots,g_{n-1}),\\
	s_j(g_0,g_1,\dots,g_n) &= (g_0,\dots,g_j,1,g_{j+1},\dots,g_n)
\end{align*}

The simplicial weight is given by $w_n(g_0,\dots,g_n) = \sum_{i=1}^n L(g_i)$. 
$N_{\bullet}^{cy}(G)$ is a $\B$-bounded simplicial set.

\begin{example}
The bar resolution, $EG_{\bullet}$.
\end{example} 
Recall that the non-homogeneous bar resolution of $G$ is  
$EG_{\bullet} = \{[n]\mapsto G^{n+1}\}_{n\geq 0}$ with

\begin{align*}
	\partial_i[g_0,\dots,g_n] &= [g_0,\dots,g_ig_{i+1},\dots,g_n],\, 0 \leq i \leq n-1,\\
	\partial_n[g_0,\dots,g_n] &= [g_0,\dots,g_{n-1}],\\
	s_j[g_0,\dots,g_n] &= [g_0,\dots,g_j,1,g_{j+1},\dots, g_n]
\end{align*}

The simplicial weight function on $EG_{\bullet}$ is given by
$w([g_0,\dots,g_n]) = \sum_{i=0}^n L(g_i)$. The left $G$-action is
given, as usual, by $g[g_0,g_1,\dots,g_n] = [gg_0,g_1,\dots,g_n]$.
Note that with respect to the given weight function and action of
$G$, $EG_{\bullet}$ is a $\B$-bounded simplicial $G$-set for any $\B$.

For $\B$ a bounding class and $(X,w)$ a weighted set, $\B C(X)$ will denote
the collection of $\B$-bounded functions on $X$.  That is, $\B C(X)$ consists
of all functions $f : X \to \C$ such that there is $\phi \in \B$ with 
$| f(x ) | \leq \phi( w(x) )$ for all $x \in X$.  Dually, $\RH_{\B,w}(X)$ will
consist of all $f : X \to \C$ such that for all $\phi \in \B$, the sum
$\sum_{x \in X} |f(x)| \phi( x )$ is finite.  This is a completion of 
the collection of finitely supported chains on $X$ with respect to a family of seminorms on 
$\RH_{\B,w}(X)$ given by
\[ \| f \|_{\phi} := \sum_{x \in X} |f(x)| \phi(w(x)) \]
for $\phi \in \B$.  When $(X,w)$ is $(G,L)$, a discrete group with a word-length function, 
$\BG$ is the $\B$-Rapid Decay algebra of $G$. When $\B$ is equivalent to a countable 
bounding class the seminorms give $\BG$ the structure of a Fr\'echet algebra.

For $(X_{\bullet}, w_{\bullet})$  a weighted simplicial set and $\B$ a bounding
class, set $\B C_n(X_\bullet) = \RH_{\B, w_n}(X_n)$, the completion of $C_n(X_\bullet)$.
If each of the face maps of $(X_\bullet, w_\bullet)$ are $\B$-bounded, then the boundary
maps $d_n : C_n(X_\bullet) \to C_{n-1}(X_\bullet)$ extend to $\B$-bounded maps
$d_n : \B C_n(X_\bullet) \to \B C_{n-1}(X_\bullet)$.  The homology of the resulting bornological complex
$\left\{\B C_n(X_\bullet), d_n \right\}_{n \geq 0}$ is the $\B$-bounded homology of $(X_\bullet, w_\bullet)$,
$\B H_*(X)$.  Similarly the cohomology of the bornological cochain complex $\left\{ \B C^n(X_\bullet) = \B C( X_n ), \partial \right\}$
is the $\B$-bounded cohomology of $(X_\bullet, w_\bullet)$, $\B H^*(X)$.
In the case of the weighted cyclic bar construction, $N^{cy}_{\bullet}(G)$ we have identifications from \cite{JOR1},
\begin{align*}
\B H_*( N^{cy}_{\bullet}(G) ) &\isom HH^{t}_*( \BG )\\
\B H^*( N^{cy}_{\bullet}(G) ) &\isom HH_{t}^*( \BG )
\end{align*}

A group with length function $(G,L)$ has $\B$-cohomological dimension ($\B$-cd) $\leq n$ if there is a projective resolution of 
$\C$ over $\BG$ of length at most $n$.  Here we do not require finite generation over $\BG$ in any degree, but as in
\cite{JOR2} we require a $\B$-bounded $\C$-linear contracting homotopy.  We denote by $\bExt_{\BG}$ the $\Ext$ functor
on the category of bornological $\BG$-modules.  That is, for a bornological $\BG$-module, $\bExt_{\BG}(\C, M )$
is obtained by taking a projective resolution of $\C$ over $\BG$, apply $\bHom( \cdot, M )$, and take the
cohomology of the resulting cochain complex.

We have the following analogue of Lemma VIII.2.1 of \cite{Br1}.
\begin{lemma}
For a group with length function $(G,L)$ and $\B$ a composable bounding class, the following are equivalent.
\begin{enumerate}
	\item $\B$-cd $(G,L) \leq n$.
	\item $\bExt^{i}_{\BG}( \C, \cdot ) = 0$ for all $i > n$.
	\item $\bExt^{n+1}_{\BG}( \C, \cdot ) = 0$.
	\item If $0 \to K \to P_{n-1} \to \cdots \to P_1 \to P_0 \to \C \to 0$ is any sequence of bornological $\BG$-modules
			admitting a $\B$-bounded $\C$-linear contracting homotopy with each $P_j$ projective then $K$ is projective.
	\item If $0 \to K \to F_{n-1} \to \cdots \to F_1 \to F_0 \to \C \to 0$ is any sequence of bornological $\BG$-modules
			admitting a $\B$-bounded $\C$-linear contracting homotopy with each $F_j$ free then $K$ is free.
\end{enumerate}
\end{lemma}
\begin{proof}
(1) iff (2) iff (3) iff (4) as in \cite{Br1} needing virtually no modification.  (4) iff (5) follows from the
bornological Eilenberg Swindle.
\end{proof}

\vspace{.2in}

\section{Avoiding bounded conjugators, and the proof of Theorem A}

There is always, for each conjugacy class $<x>\in <G>$, an isomorphism of simplicial sets
\[ G_x \backslash G \underset{G}{\times} EG_{\bullet} \to N^{cy}_{\bullet}(G)_x. \]
(Here $G_x$ is the centralizer of the element $x \in G$.)
This map induces an isomorphism
\[ 
\underset{<x> \in <G>}{\coprod} G_x \backslash G \underset{G}{\times} EG_{\bullet} \to \underset{<x> \in <G>}{\coprod} N^{cy}_{\bullet}(G)_x. 
\]
When $x$ is not central in $G$, this isomorphism depends on the particular choice of representative for
each conjugacy class.  Recall from \cite{JOR1} that a conjugacy class $<x>$ has a {\underline{$\B$-bounded conjugator
length}} if there is an $f \in \B$ such that, for all $x, y \in S_{<x>} =$ the set of elements in $G$ conjugate to $x$, there is a $g \in G$ with $g^{-1} x g = y$
and $L_G( g ) \leq f( L_G(x) + L_G(y) )$.  This is equivalent to the
statement that the natural map
\[ 
\pi_x : G_x \backslash G \to S_x 
\] 
defined by $G_x g \mapsto g^{-1} x g$ is a $\B$-bounded isomorphism. When $<x> \in <G>$ has a $\B$-bounded conjugator length, the
map $G_x \backslash G \underset{G}{\times} EG_{\bullet} \to N^{cy}_{\bullet}(G)_x$ is a $\B$-bounded isomorphism of weighted simplicial sets, 
which in turn induces an isomorphism of cohomology groups
\[ 
HH^*(\BG)_{<x>} \isom {\B}H^*(BG_x).
\]

We are interested in the case where there is no such conjugacy bound (which is most of the time). In this case, the map
\[
\underset{{<x> \in <G>} }{\coprod}\left(G_x \backslash G \underset{G}{\times} EG_{\bullet}\right) \to N^{cy}_{\bullet}(G)
\] 
still exists as a $\B$-bounded map and an isomorphism of simplicial sets, however a choice of $\B$-bounded inverse may not exist.

Assume given a type $\B$ simplicial resolution $\Gamma_{\bullet} \surj G$.
Associated to this simplicial type resolution is a simplicial chain complex equipped with augmentation map
\[ 
{\B}C_*(N^{cy}_{\bullet}( \Gamma_{\bullet} )) := \{[n]\mapsto {\B}C_*(N^{cy}_{\bullet}(\Gamma_n))\}_{n\ge 0} \surj {\B}C_*(N^{cy}_{\bullet}(G)). 
\]
By Theorem 1 of \cite{O2}, this simplicial complex is of resolution type. Consequently, the double complex associated to ${\B}C_*(N^{cy}_{\bullet}( \Gamma_{\bullet} ))$ maps via the augmentation map to ${\B}C_*(N^{cy}_{\bullet}(G))$ by a map inducing an isomorphism in $\B$-bounded homology and cohomology.

Now consider the following diagram of simplicial sets:
\[
\xymatrix{
 \ar[d] \ar@<-1ex>[d] \ar@<-2ex>[d]& \\
 N^{cy}_{\bullet}(\Gamma_1) \ar[d] \ar@<-1ex>[d] \ar@<-1ex>[u] \ar@<-2ex>[u] & \\
N^{cy}_{\bullet}(\Gamma_0) \ar@<-1ex>[u] \ar@{>>}[d] \ar@{<-}[r]_-{\isom}^-{\phi} & \underset{{y \in <\Gamma_0>}}{\coprod} (\Gamma_{0})_y \backslash \Gamma_0 \underset{\Gamma_0}{\times} E(\Gamma_0)_{\bullet} \ar@{>>}[d]  \\
N^{cy}_{\bullet}(G) & \underset{{<x> \in <G>}}{\coprod} G_{x} \backslash G \underset{G}{\times} EG_{\bullet} \ar[l]
}
\]
For the map on the right, explicit centralizer subgroups $G_x$ and $(\Gamma_0)_y$ are chosen as follows: for each $<x>\in <G>$, fix a basepoint $x$ of the set $S_{<x>}$. Then for each $<y>\in <\Gamma_0>$, we choose a basepoint $y\in S_{<y>}$ of minimal word-length so that the augmentation map $\varepsilon:\Gamma_0\surj G$ induces a map of basepointed sets
\[
\left(S_{<y>},y\right)\overset{\varepsilon}{\longrightarrow}\left(S_{<\varepsilon(y)>},\varepsilon(y)\right)
\]
The map $\phi$ is a linearly bounded isomorphism with linearly bounded inverse (hence a $\B$-bounded isomorphism for all $\B$), as $\Gamma_0$ is a free group equipped with word-length metric. 
This isomorphism will allow us to construct a simplicial resolution of $\underset{{<x> \in <G>}}{\coprod} G_x \backslash G \underset{G}{\times} EG_{\bullet}$.

Define the bisimplicial set $X_{\bullet \bullet}$ to be:
\[ 
X_{\bullet n} = \begin{cases} 
\underset{{<y> \in <\Gamma_0>}}{\coprod} (\Gamma_0)_y \backslash \Gamma_0 \underset{\Gamma_0}{\times} E(\Gamma_0)_\bullet , & n = 0\\
N^{cy}_{\bullet}( \Gamma_n ),  & n > 0 
\end{cases} 
\]
The face maps and degeneracy maps between $X_{\bullet n}$ and $X_{\bullet m}$, for $m,n > 0$ are given by the corresponding maps $\{\partial^{\Gamma_{\bullet}}_i,s^{\Gamma_{\bullet}}_j\}$ in $N^{cy}_\bullet( \Gamma_\bullet )$ induced by the face and degeneracy maps of $\Gamma_{\bullet}$.  The face and
degeneracy maps involving $X_{\bullet 0}$ are determined by 
$\partial_i : X_{\bullet 1} \to X_{\bullet 0}$, $\partial_i = \phi^{-1} \circ \partial^{\Gamma_{1}}_i$ for $i = 0,1$
and $s_0 : X_{\bullet 0} \to X_{\bullet 1}$, $s_0 = s^{\Gamma_{0}}_0 \circ \phi$.

These maps satisfy the simplicial identities and are all $\B$-bounded.  In particular we have a
$\B$-bounded bisimplicial $G$ set $\{ [n] \mapsto X_{\bullet n} \}$.  There is a $\B$-bounded isomorphism of bisimplicial $G$-sets 
\[ 
\{ [n] \mapsto X_{\bullet n} \}_{n \geq 0} \to \{ [n] \mapsto N^{cy}_{\bullet}( \Gamma_n ) \}_{n \geq 0} 
\]

As the augmentation $\varepsilon: \Gamma_0 \to G$ is surjective, the induced map
\[ 
\varepsilon_0: \coprod_{y \in <\Gamma_0>} (\Gamma_0)_y \backslash \Gamma_0 \underset{\Gamma_0}{\times} E(\Gamma_0)_{\bullet}
\to \coprod_{<x> \in <G>} G_x \backslash G \underset{G}{\times} EG_{\bullet}
\]
is a surjective morphism of weighted simplicial sets, which moreover is bounded. Our aim is to show that the induced and given weight functions on the set to the right are equivalent. To this end, we define a section $\tilde{s}$ of $\varepsilon_0$ by using a set-theoretic splitting $G \inj\Gamma_0$ of
the augmentation $\varepsilon:\Gamma_0\surj G$ which is length-minimizing among all basepoint-preserving maps; precisely, if $x\in S_{<x>}$ is the basepoint of $S_{<x>}$, then i) $\tilde{s}(x)\in S_{<\tilde{s}(x)>}$ should be the basepoint of $S_{<\tilde{s}(x)>}$, and ii) $\tilde{s}(x)$ should have minimal word-length among all basepoints $y\in S_{<y>}\subset\Gamma_0$ for which $\varepsilon(y) = x$. Because of the manner in which the basepoints in $\Gamma_0$ were chosen, such a section always exists, and moreover it can be chosen so as to be length-preserving on basepoints. Of course, it can always be chosen so as to be length-preserving away from basepoints.

\begin{lemma}\label{lem:BddSplitting}
The section $\tilde{s}$ induces a splitting $\{\tilde{s}_n\}_{n\ge 0}$ of $\varepsilon_0$ on the level of graded sets which is $\B$-bounded.
\end{lemma}
\begin{proof}
Denote by $L_0$ the proper length function on $\Gamma_0$ and by $L$ the proper length function
on $G$.  Fix a conjugacy class $<y> \in <\Gamma_0>$ and consider an element of
$(\Gamma_0)_y \backslash \Gamma_0 \underset{\Gamma_0}{\times} E(\Gamma_0)_{\bullet}$, $[(\Gamma_0)_y \gamma \times (\gamma_0, \gamma_1, \ldots, \gamma_n)]$.
We may assume $(\Gamma_0)_y \gamma \times (\gamma_0, \gamma_1, \ldots, \gamma_n)$ is a minimal weight
representative of the class, and that $\gamma$ is a minimal length element of the coset $(\Gamma_0)_y\gamma$,
so the weight of $[(\Gamma_0)_y \gamma \times (\gamma_0, \gamma_1, \ldots, \gamma_n)]$ is
$L_0(\gamma) + L_0( \gamma_0 ) + \cdots + L_0(\gamma_n )$.

Given that we have arranged the basepoints for the sets $\{S_{<y>}\}$ so as to make the augmentation basepoint-preserving, the induced map $\epsilon_0$ on the right has the form
\[
\epsilon_0[(\Gamma_0)_y \gamma \times (\gamma_0, \gamma_1, \ldots, \gamma_n)] =
[G_x \epsilon(\gamma) \times ( \epsilon(\gamma_0), \epsilon(\gamma_1), \ldots, \epsilon(\gamma_n))]
\]
where $x = \varepsilon(y)$.
The weight of this class is no more than $L( \epsilon(\gamma) ) + L(\epsilon(\gamma_0) ) + \cdots
+ L(\epsilon(\gamma_n) )$. If $\beta \in \B$ is a bounding function for $\epsilon$, then we see
that the weight of this class is no more than $\beta( L_0( \gamma ) ) + \beta( L_0( \gamma_0 ) ) + \cdots
+ \beta( L_0(\gamma_n))$. Thus $\epsilon_0$ is bounded by $(n+2)\beta \in \B$.

 For each $n$-simplex $[G_x g \times (g_0, g_1, \ldots, g_n)]
\in G_x \backslash G \underset{G}{\times} EG_n$, fix a minimal weight representative $G_x g \times (g_0, g_1, \ldots, g_n)$
with $g$ of minimal length in the coset $G_x g$. 
The desired section on $n$-simplices is defined by
\[ 
\tilde{s}_n [G_x g \times (g_0, g_1, \ldots, g_n)] = [ (\Gamma_0)_{\tilde{s}(x)} \tilde{s}(g) \times ( \tilde{s}(g_0), \tilde{s}(g_1), \ldots, \tilde{s}(g_n))].
\]
Note that the collection of maps $\{\tilde{s}_n\}_{n\ge 0}$ will not, in general, define a map of simplicial sets, because the original section $\tilde{s}$ cannot be chosen to be a homomorphism unless $G$ itself is free. However, it is clear that $\tilde{s}_n$ is a set-theoretic splitting of $\epsilon_0$ on $n$-simplicies which moreover is bounded. To see this, note that the weight of $\tilde{s}_n [G_x g \times (g_0, g_1, \ldots, g_n)]$ can be no greater than $L_0( \tilde{s}(g) ) +
L_0( \tilde{s}(g_0) ) + \cdots + L_0(\tilde{s}(g_n) )$. Thus if $\tilde{s}$ is bounded by $\beta' \in \B$, then
$\tilde{s}_n$ is bounded by $(n+2)\beta' \in \B$.
\end{proof}

As we have indicated, there are two basic approaches to endowing $\underset{{<x> \in <G>}}{\coprod} G_x \backslash G \underset{G}{\times} EG_{\bullet}$ with a weight function.
One could use the weight induced by the length structure on $G$ itself, or one could use
the weights induced by the surjection $\epsilon_0$. However, the existence of a bounded section guaranteed by the previous lemma shows that these two
methods yield $\B$-equivalent weight structures.

\begin{theorem}\label{thm:Res}
If $\Gamma_{\bullet} \surj G$ is a type $\B$  resolution of $G$,
then $X_{\bullet\bullet}$ is a $\B$-bounded augmented simplicial set for which the bisimplicial 
set $\{ [n] \to X_{\bullet,n} \}_{n \geq 0}$ yields isomorphisms in homology and cohomology
\begin{gather*}
{\B}H^*\left( \coprod_{<x> \in G} G_x \backslash G \underset{G}{\times} EG_\bullet\right)\isom {\B}H^*(X_{\bullet\bullet})\\
{\B}H_*\left(\underset{{<x> \in G}}{\coprod} G_x \backslash G \underset{G}{\times} EG_\bullet\right)\isom {\B}H_*(X_{\bullet\bullet})
\end{gather*}
\end{theorem}
\begin{proof}
The sections coming from the previous lemma, together with those arising from the type $\B$ resolution, imply that the augmented simplicial topological chain complex
\[
[n]\mapsto
\begin{cases}
{\B}C_*(X_{\bullet n})\quad n\ge 0\\
{\B}C_*\left(\underset{{<x> \in G}}{\coprod} G_x \backslash G \underset{G}{\times} EG_\bullet\right)\quad n = -1
\end{cases}
\]
is of resolution type. The result then follows as in \cite[Theorem 2]{O2}.
\end{proof}

\begin{corollary}
There are isomorphisms in $\B$-bounded homology and cohomology 
\begin{gather*}
{\B}H^*\left(\underset{<x> \in <G>}{\coprod} G_x \backslash G \underset{G}{\times} EG_\bullet\right) \isom {\B}H^*(N^{cy}_{\bullet}(G)) = H_t^*(\BG)\\
{\B}H_*\left(\underset{<x> \in <G>}{\coprod} G_x \backslash G \underset{G}{\times} EG_\bullet\right) \isom {\B}H_*(N^{cy}_{\bullet}(G)) = H^t_*(\BG)
\end{gather*}
\end{corollary}
\begin{proof}
The weighted bisimplicial set $X_{\bullet \bullet}$ is $\B$-boundedly isomorphic to the weighted bisimplicial set $\{ [n] \to N^{cy}_{\bullet}(\Gamma_n)\}_{n \geq 0}$, and the augmentation map 
$\varepsilon: {\B}C_*(N^{cy}_{\bullet}(\Gamma_{\bullet}))\surj {\B}C_*(N^{cy}_{\bullet}(G))$ induces an isomorphism in $\B$-bounded homology and cohomology. The result follows then from Theorem \ref{thm:Res}.
\end{proof}

\begin{corollary}\label{cor:GlobalIsom}
There is an isomorphism in $\B$-bounded cohomology
\[ 
HH^*_t(\BG) \isom {\B}H^*\left( \coprod_{<x> \in <G>} BG_x\right). 
\]
\end{corollary}
\begin{proof}
By \cite[Prop. 1.4.5]{JOR1} there is an isomorphism 
\[
{\B}H^*\left(\underset{<x> \in <G>}{\coprod} BG_x \right) \isom {\B}H^*\left( \underset{<x> \in <G>}{\coprod} G_x \backslash G \underset{G}{\times} EG_\bullet\right).
\]
where each centralizer subgroup $G_x$ is equipped with the induced word-length function coming from the embedding into $G$.
The previous corollary gives the result.
\end{proof}

\begin{corollary}\footnotemark
For each non-elliptic conjugacy class $<x> \in <G>$ with $\Dist( (x) ) \leq \B$, there is an isomorphism
\[ 
HC^*_t(\BG)_{<x>} \isom {\B}H^*(G_x/(x))
\]
where $G_x/(x)$ is equipped with the word-length function induced by the projection $G_x\surj N_x = G_x/(x)$.
\end{corollary}
\footnotetext{This result generalizes Theorem 2.4.4 of \cite{JOR1}}
\begin{proof}
The isomorphism of Corollary \ref{cor:GlobalIsom} splits over conjugacy classes
\[ 
HH^*_t(\BG)_{<x>} \isom {\B}H^*( BG_x ). 
\]
When $\Dist((x)) \leq \B$, the Connes-Gysin sequences in \cite[\S 1.4]{JOR1} for the summand $HC^*_t(\BG)_{<x>}$ along with this isomorphism for
$HH^*( \BG)_{<x>}$ gives the result.  
\end{proof}
In the case $\Dist( (x) ) > \B$, it is almost certainly true that $HC^*_t(\BG)_{<x>} \isom {\B}H^*(G_x) \tensor HC^*(\C)$.
To conclude this, however, one needs to know that if $\Z$ is embedded into a group with distortion greater than $\B$, then
${\B}H^*(\Z)$ is trivial for $*\geq 1$ (where $\mathbb Z$ is equipped with the word-length function induced by the embedding).  This is conjectured to be true, but is currently unverified.

\begin{nnremark} As noted, the results of this section were previously known only for groups with conjugacy classes satisfying $\B$-bounded conjugacy length bounds of \cite{JOR1}. They may be interpreted as saying that, even when the conjugacy problem for $G$ cannot be solved in an appropriately bounded timeframe, or even solved at all, it can always be solved ``up to bounded homotopy''.
\end{nnremark}
\vspace{.5in}

\section{A class of groups satisfying $\B$-SrBC}

We start with a technical lemma.

\begin{lemma}\label{lem:BCohDimExts}
Suppose $N \inj G \surj Q$ is an extension of groups with length functions, and let $\B$
be a composable bounding class.  If $\B\textrm{-cd }N < \infty$ and $\B\textrm{-cd }Q < \infty$ then $\B\textrm{-cd }G \leq \B\textrm{-cd }N + \B\textrm{-cd }Q$.
\end{lemma}
\begin{proof}
We follow the setup of \cite[Theorem 1.1.12]{O1}.
Denote by $L_G$ the length function on $G$, which by restriction is the length function on $N$, 
and by $L_Q$ the quotient length function on $Q$. As $\B\textrm{-cd }Q < \infty$, there is a finite length 
free resolution of $\C$ over $\RH_{{\B},L_Q}(Q)$, with $\B$-bounded $\C$-linear splittings.  Denote 
this resolution by $R_*$, let $P_*$ be a free resolution of $\C$ over $\RH_{{\B},L_G}(G)$, and set
$S_q = \C \btensor_{\RH_{\B,L_G}(N)}P_q$.

Fix an $\RH_{\B,L_G}(G)$-module $M$ and let $C^{p,q} = \B\Hom_{\RH_{\B,L_Q}(Q)}( R_p \btensor S_q, M )$.
The spectral sequence associated to the row filtration collapses at $E_2^{*,*}$ with $E_2^{0,q} \isom \B H^{q}( G; M )$
and $E_2^{p,q} = 0$ if $p > 0$.

Consider the spectral sequence associated to the column filtration.  
For $p > \B\textrm{-cd }Q$, $C^{p,q} = 0$, so $E_1^{p,q} = 0$ for $p > \B\textrm{-cd }Q$.  Similarly,
by the finiteness condition on $\B\textrm{-cd }N$, $E_1^{pq} = 0$ for $q > \B\textrm{-cd }N$.
Consequently, $E_2^{p,q} = 0$ whenever $p+q > \B\textrm{-cd }Q + \B\textrm{-cd }N$.
As a consequence, $\B H^n(G;M) = 0$ for $n > \B\textrm{-cd }Q + \B\textrm{-cd }N$ for all coefficients $M$.
\end{proof}

Let $\mathcal{C}$ be the collection of all countable groups $G$ which satisfy the nilpotency condition, and 
let $\B\textrm{-}\mathcal{C}$ be the collection of all groups with length function $(G,L)$ which satisfy
the $\B$-nilpotency condition.  Any group $(G, L) \in \B\textrm{-}\mathcal{C}$ satisfies $\B$-SrBC, just as
any $G \in \mathcal{C}$ satisfies SBC.  Moreover if $(G,L)$ satisfies $\B$-SrBC, then $G$ satisfies
SrBC.  Recall that a group with length function $(G, L)$ is said to be \underline{$\B$-isocohomological}($\B$-IC)
if the inclusion $\CG \to \BG$ induces isomorphisms $\B H^p(G; \C) \to H^p(G;\C)$ for all $p$.  The
group is \underline{$\B$-strongly isocohomological}($\B$-SIC) if the maps $\B H^p(G; M) \to H^p(G;M)$
are isomorphisms for every $p$ and every bornological $\BG$-module $M$.
\begin{theorem}
Suppose the discrete group $G$ lies in $\mathcal{C}$ and $L$ is a proper length function on $G$.
If for each non-elliptic conjugacy class $<x> \in <G>$, the centralizer $(G_x, L)$ is $\B$-IC and
the embedding $\Z \isom (x) \inj G_x$ is at most $\B$-distorted, then $(G,L) \in \B\textrm{-}\mathcal{C}$.
\end{theorem}
\begin{proof} 
By Theorem A $HC^*(\CG)_{<x>} \isom HC^*_{t}( \BG )_{<x>}$.  This isomorphism identifies
periodicity operators.
\end{proof}
As shown in the introduction this class contains all semihyperbolic groups $G$ with word-length
function $L$, which satisfy the nilpotency condition.
We remark that this, in turn, contains all word-hyperbolic and finitely generated abelian groups.

We now identify and examine a class of groups inside $\B\textrm{-}\mathcal{C}$.
\begin{definition}
Let $\BE$ be the collection of all countable groups with length function $(G,L)$ which
satisfy the following two properties.
\begin{enumerate}
	\item $G$ has finite $\B$-cd.
	\item For every non-elliptic conjugacy class $<x> \in <G>$, $N_x$ has finite $\B$-cd
		in the induced length function.
	\item For every non-elliptic conjugacy class $<x> \in <G>$, the distortion of $(x)$
		as a subgroup of $G$ is $\B$-bounded.
\end{enumerate}
Let $\E$ be the collection of all countable groups in $\B_{max}\textrm{-}\E$.
\end{definition}
It is clear that $\BE \subset \B\textrm{-}\mathcal{C}$ and $\E \subset \mathcal{C}$.

The following is clear from the definition of $\B$-SIC.
\begin{lemma}
Suppose $G$ is a countable group in $\E$ with length function $L$.  If $(G,L)$ is $\B$-SIC and 
for each non-elliptic $<x> \in <G>$, $N_x$ is $\B$-SIC, then $(G,L)$ is in $\BE$.
\end{lemma}

\begin{theorem}
The class $\BE$ is closed under the following operations when considered in the induced length functions.
	\begin{enumerate}
		\item Taking subgroups.
		\item Taking extensions of groups with length functions.
		\item Acting on trees with vertex and edge stabilizers in $\BE$.
	\end{enumerate}
\end{theorem}
\begin{proof}
We follow the proof of Theorem 4.3 of \cite{Ji1}.
(1) follows from \cite{Ji1} after noting that if $(G,L)$ is a group with length function and $H < G$, any projective resolution of 
$\C$ over $\BG$ is also a projective resolution of $\C$ over $\RH_{\B,L}(H)$.

(2) follows as in \cite{Ji1} using Lemma \ref{lem:BCohDimExts} to bound the $\B$-cd of the involved extensions.

(3) follows as in \cite{Ji1}, noting that each of the stabilizers involved are assumed to be in $\BE$ in the restricted length function.
\end{proof}

\vspace{0.5in}



\end{document}